\documentclass[11pt]{article}
\usepackage{amssymb}
\usepackage{latexsym}
\usepackage{amsfonts}
\usepackage{amsmath}
\usepackage{setspace}
\usepackage[hypertexnames=false,colorlinks=true, linkcolor=blue, urlcolor=red, citecolor=green]{hyperref}
\newtheorem{con0}{Theorem}[section]
\newtheorem{con1}[con0]{Condition}
\newtheorem{thrm}{Theorem}[section]
\newtheorem{lemma}[thrm]{Lemma}
\newtheorem{prop}[thrm]{Proposition}

\newtheorem{cor}[thrm]{Corollary}
\newtheorem{remark}[thrm]{Remark}
\newtheorem{exam}{Example}

\numberwithin{equation}{section}
\usepackage[dvips]{color}
\usepackage{mathtools}
\usepackage{mathrsfs}
\usepackage{comment}
\usepackage{xcolor}

\def\proof{\noindent{\it Proof.~}}
\def\qed{\hfill$\square$\smallskip}

\def\bgcondition{\begin{con1}}\def\edcondition{\end{con1}}

\def\blemma{\begin{lemma}}\def\elemma{\end{lemma}}
\def\bproposition{\begin{proposition}}\def\eproposition{\end{proposition}}
\def\btheorem{\begin{theorem}}\def\etheorem{\end{theorem}}
\def\bcorollary{\begin{corollary}}\def\ecorollary{\end{corollary}}
\def\bremark{\begin{remark}}\def\eremark{\end{remark}}
\def\bcondition{\begin{condition}}\def\econdition{\end{condition}}

\def\benumerate{\begin{enumerate}}\def\eenumerate{\end{enumerate}}
\def\bitemize{\begin{itemize}}\def\eitemize{\end{itemize}}

\def\beqlb{\begin{eqnarray}}\def\eeqlb{\end{eqnarray}}
\def\beqnn{\begin{eqnarray*}}\def\eeqnn{\end{eqnarray*}}

\def\qqquad{\qquad\qquad}
\def\proof{\noindent{\it Proof.~~}}
\def\qed{\hfill$\Box$\medskip}

\def\<{\langle}\def\>{\rangle}
\def\mcr{\mathscr}\def\mbb{\mathbb}\def\mbf{\mathbf}\def\mrm{\mathrm}

\makeatletter
\begin{document}
	\allowdisplaybreaks

	\title{\Large \bf Coupling for one-dimensional subcritical and critical CBI processes with jumps}
\author{ \bf
	Shukai Chen
	\hspace{1mm}\hspace{1mm} and \hspace{1mm}\hspace{1mm}
	Chunhua Ma
	}
	\date{}
	\maketitle
	
	\begin{abstract}
		 We develop a cluster representation for one-dimensional CBI processes with jumps. Using this representation, we establish total variation convergence under conditions on the branching L\'evy measure. In the subcritical case, we obtain polynomial and exponential rates under distinct regularity assumptions, while the strong Feller property is also established. In the critical case, we derive an explicit bound involving the cumulant integral. Our proofs use a coupling method that has proved effective for establishing ergodicity of Ornstein--Uhlenbeck processes.
	\end{abstract}
	
	\medskip
	

	\medskip
	

\section{Introduction and Main results}

\subsection{Introduction}

Let $\sigma\ge0$ and $\beta$ be constants, and suppose that $(\xi\wedge\xi^2)\nu(d\xi)$ is a finite measure on $(0,\infty)$. For $\lambda\ge 0$, define
\[
\Psi(\lambda) = -\beta\lambda + \frac{1}{2}\sigma^2 \lambda^2 + \int_0^\infty(e^{-\lambda \xi}-1+\lambda\xi)\nu(d\xi).
\]
A Markov process with state space $\mbb{R}_+=[0,\infty)$ is called a \textit{continuous-state branching process} (CB process) with branching mechanism $\Psi$ if its transition semigroup $(Q_t)_{t\ge0}$ satisfies
\begin{align}\label{CB}
\int_0^\infty e^{-\lambda y}Q_t(x,dy)
 =
e^{-x v_t(\lambda)},
\end{align}
where $t\mapsto v_t(\lambda)$ is the unique nonnegative solution to
\begin{align}\label{2.2}
\frac{\partial}{\partial t}v_t(\lambda) = -\Psi(v_t(\lambda)), \qquad
v_0(\lambda)=\lambda.
\end{align}
The CB process is called \textit{critical}, \textit{subcritical}, or \textit{supercritical} according to whether $\beta=0$, $\beta<0$, or $\beta>0$, respectively.
The semigroup $(Q_t)_{t\ge0}$ is Feller and therefore admits a Hunt realization. Let $X = (\Omega, \mcr{G}, \mcr{G}_t, X_t,
\mathbb{P}_x)$ be a Hunt realization of the CB process. The hitting time
$\tau_0 = \inf\{t\ge 0: X_t=0\}$ is called the \textit{extinction time}
of $X$. It follows from Li \cite[Theorem 3.5]{Li11} that for all $t\ge 0$, the
limit
\[
\bar{v}_t = \uparrow \lim_{\lambda\to \infty} v_t(\lambda)
\]
exists in $(0,\infty]$, and
\[
\mathbb{P}_x(\tau_0\le t) = \mathbb{P}_x(X_t=0) = \exp\{-x\bar{v}_t\}.
\]
Consequently, $\mathbb{P}_x(\tau_0< \infty)=\exp\{-x\bar{v}\}$,
where $\bar{v}:=\lim_{t \rightarrow +\infty}\bar{v}_{t}\in [0, +\infty]$.
Grey \cite{G74}, see also Li \cite{Li11}, established the following result.
\begin{prop}
$\bar{v}_t< \infty$ for all $t>0$ if
and only if
there exists a constant $\theta> 0$ such that
$\Psi(\lambda)>0$ for all $\lambda> \theta$ and
\begin{align}\label{t2.1}
\int_{\theta}^\infty \Psi(\lambda)^{-1}d\lambda< \infty.
\end{align}
Moreover, if $\eqref{t2.1}$ holds, then for any $x>0$, we have $\mathbb{P}_x(\tau_0< \infty)=1$ (i.e., $\bar{v}=0$) if and only if $\beta\leq0$.
\end{prop}

Condition \eqref{t2.1} is commonly known as {\it Grey's condition}. A necessary condition for \eqref{t2.1} is either $\sigma>0$ or $\int_0^1\lambda\nu(d\lambda)=\infty$.
Now let $b\geq0$ be a constant and  $(1\wedge\xi)\,m(d\xi)$ be
a finite measure on $(0,\infty)$. Define the function $\Phi$ on $[0,
\infty)$ by
\[
\Phi(\lambda)=b\lambda+\int_0^\infty (1-e^{-\lambda\xi})\,m(d\xi).
\]
We refer to $\Phi$ as the {\it immigration mechanism}. Let $t\mapsto v_t(\lambda)$ be defined as in \eqref{2.2}. A Markov process
with state space $\mbb{R}_+$ is called a \textit{continuous-state branching process with immigration} (CBI process) with branching mechanism $\Psi$ and immigration mechanism $\Phi$
if its transition semigroup $(P_t)_{t\ge 0}$ satisfies
 \begin{align}\label{transition semigroup}
 \int_0^\infty e^{-\lambda y}P_t(x,dy)
 =
\exp\Big\{-xv_t(\lambda) - \int_0^t\Phi(v_s(\lambda))ds\Big\}.
 \end{align}
When $\Phi\equiv0$, a CBI process reduces to a CB process.

In this paper, we consider the case that $\sigma=0$, $m\equiv0$, and $b>0$. Then the mechanisms take the form
\begin{align}\label{mechanism}
\Psi(\lambda) = -\beta\lambda+ \int_0^\infty(e^{-\lambda \xi}-1+\lambda\xi)\nu(d\xi),\quad \Phi(\lambda)=b\lambda,\quad \lambda\ge0.
 \end{align}
Let $N(ds,du,d\xi)$ be a Poisson random measure on $(0,\infty)^3$ with intensity $dsdu\nu(d\xi)$. Then, for each $x\ge 0$, there
is a pathwise unique nonnegative strong solution $\{X^x_t:t\ge 0\}$ to the following stochastic
equation:
 \begin{align}\label{CBI SDE}
X^x_t=
x + \int_0^t(b+\beta X^x_s) ds
+\int_0^t\int_0^{X^x_{s-}}\int_0^\infty \xi \tilde{N}(ds,du,d\xi),
 \end{align}
where $\tilde{N}(ds,du,d\xi)=N(ds,du,d\xi)-dsdu\nu(d\xi)$. The solution is a CBI process; we refer to Dawson and Li \cite{DaL06,DaL12} or Li and Ma \cite{LiM08} for further details. This construction leads to the following useful property.
\begin{thrm}\label{branching property}\;
(Dawson and Li \cite[Theorem 1.2, Proposition 1.3]{DaL12}) For any $x\geq y\geq0$, we have almost surely that
\[
X^x_t\geq X^y_t ~\text{for~all}~ t\geq0,
\]
and the process $\{X^x_t-X^y_t:t\geq0\}$ is a CB process. Moreover, $\{X^x_t-X^y_t:t\geq0\}$ is independent of $\{X^y_t:t\geq0\}$.
\end{thrm}

A pair $\{(X_t,Y_t):t\geq0\}$ is called a {\it coupling} of the Markov process associated with a given transition probability if both $\{X_t:t\geq0\}$ and $\{Y_t: t\geq0\}$ are Markov processes governed by the same transition probability  (though possibly with different initial distributions).  In this case, $\{X_t:t\geq0\}$ and $\{Y_t: t\geq0\}$ are referred to as the {\it marginal processes} of the coupling.  The coupling $\{(X_t,Y_t):t\geq0\}$  is said to be {\it successful} if the {\it coupling time}
\[
 T:=\inf\{t\geq0:X_t=Y_t\}
\]
is almost surely finite.  A Markov process is said to have a {\it coupling property} (or to admit successful coupling),  if for any two initial distributions $\mu_1$ and $\mu_2$, there exists a successful coupling whose marginal processes start from $\mu_1$ and $\mu_2$, respectively. For a systematic treatment of coupling methods, we refer to Chen \cite{Chen04} and Lindvall \cite{L92}; see also Wang \cite{Wang12} for an excellent survey.

Using Theorem \ref{branching property} and condition \eqref{t2.1}, Li and Ma \cite{LM15} constructed a successful coupling for subcritical and critical CB and CBI processes and proved the following result.

\begin{thrm}\label{coupling1}(Li and Ma \cite[Theorem 2.5]{LM15}) Under  \eqref{t2.1} with $\beta\leq0$, the transition semigroup
$(P_t)_{t\ge 0}$ defined in \eqref{transition semigroup} possesses the strong Feller property.
Moreover, for any $t>0$ and $x,y\in\mbb{R}_+$,
\begin{align}\label{estimates1}
\big\|P_t(x,\cdot)-P_t(y,\cdot)\big\|_{\rm Var}\le 2(1-e^{-\bar{v}_t|x-y|}),
\end{align}
which tends to $0$ as $t\rightarrow\infty$. Here, $\|\cdot\|_{\rm Var}$ denotes the total variation norm. In this case, the CBI process admits a successful coupling.
\end{thrm}

In the coupling construction described above, Condition \eqref{t2.1} is essential because, by Theorem \ref{branching property}, the coupling time is exactly the extinction time of the CB process $\{X^x_t-X^y_t:t\geq0\}$.
If \eqref{t2.1} is not satisfied, then the CB  process almost surely never hits $0$, meaning that $\bar{v}_t=\infty$ for all $t>0$, and thus \eqref{estimates1} becomes trivial. Furthermore, since $\bar{v}_t\leq e^{\beta(t-r)}\bar{v}_r$ for any $t\geq r>0$ (see, for example, Li \cite[Corollary 3.11]{Li11}), Li and Ma \cite{LM15} further proved exponential ergodicity in total variation for subcritical CBI processes under \eqref{t2.1}.

Unfortunately, these restrictions in Li and Ma \cite{LM15} exclude several relevant models, including Neveu's CB process with branching mechanism $\Psi_{Nev}(\lambda)=\lambda \log \lambda$, and the model of Li et al. \cite[Example 1.1]{LLWZ23} with branching L\'evy measure
\begin{align}\label{example of branching mech}
\nu(d\xi) = r \mathbf{1}_{(u,v)}(\xi) d\xi
\end{align}
for some $r > 0$ and $0 \leq u \leq v \leq 1$. The assumptions of Li and Wang \cite{LW20} likewise fail for the branching mechanism in  \eqref{example of branching mech}. This leads to the following question.

\smallskip

{\bf Question~1:} {\it Is \eqref{t2.1} necessary for the ergodicity of classical CBI processes?}

\smallskip

Additional results concerning the ergodic properties of CB-type processes have also been established. Among others, Bao and Wang \cite{BW23}, Chen and Li \cite{CL23}, and Jin et al. \cite{FJR20} proved exponential ergodicity via a coupling approach for affine processes, a natural extension of CBI processes. The Foster--Lyapunov criteria, developed by Meyn and Tweedie \cite{MT93}, were adopted in the work of Barczy et al. \cite{BDLP14} and Jin et al. \cite{JKR17} for affine processes. Roughly speaking, the Foster--Lyapunov criteria and coupling methods constitute two primary approaches for establishing ergodic properties of Markov processes. Applying the former requires irreducibility of certain skeleton chains as a key ingredient. In comparison, the coupling approach offers several advantages; notably, it allows explicit estimation of ergodic convergence rates. For nonlinear CB processes, Li and Wang \cite{LW20} obtained exponential ergodicity in both the $L^1$-Wasserstein and total variation distances. Li et al. \cite{LLWZ23} established a weighted total-variation estimate for CBI processes with competition.

It should be emphasized that all existing ergodicity proofs for CBI-type processes rely on subcriticality or on a comparable dissipativity condition. They therefore do not directly cover the classical critical CBI setting of the present paper. This motivates the second question.

\smallskip

{\bf Question~2}: {\it Can we obtain a coupling rate for critical CBI processes?}

\subsection{Main results}

To address these questions, we impose the following conditions on the branching L\'evy measure $\nu$. For $x_0\in\mathbb R$ and $\varepsilon>0$, let $B(x_0,\varepsilon)=\{x:|x-x_0|\leq\varepsilon\}$.
\begin{con1}\label{condition absolute}
There exist $x_0 \in \mathbb{R}_+$ and $\varepsilon > 0$ such that $B(x_0, \varepsilon) \subseteq (0, \infty)$, and the L\'evy measure $\nu$ has an absolutely continuous part in $B(x_0, \varepsilon)$, i.e.,
\begin{align}\label{negative density}
\mbf{1}_{B(x_0,\varepsilon)}(\xi)\nu(d\xi)=q(\xi)d\xi
\end{align}
for some non-negative function  $q\in L^\infty(B(x_0,\varepsilon))$ and
$$
\int_{B(x_0,\varepsilon)}q(\xi)^{-1}\,d\xi<\infty.
$$
\end{con1}

\begin{thrm}\label{main theorem}
Assume that $\beta\leq0$ and Condition \ref{condition absolute} holds. Then there is a constant $C_1>0$ such that, for all $x,y\in\mbb R_+$ and $t>1$, the following estimates hold:

\begin{enumerate}
\item[(i)] In the subcritical case $\beta<0$,
\[
\|P_t(x,\cdot)-P_t(y,\cdot)\|_{\rm Var}
\leq\frac{C_1(1+|x-y|)}{\sqrt{t}}.
\]
\item[(ii)] In the critical case $\beta=0$,
\[
\|P_t(x,\cdot)-P_t(y,\cdot)\|_{\rm Var}
\leq
\frac{C_1}{\sqrt{t}}
\left(1+|x-y|\int_0^t v_s(1)\,ds\right),
\]
where $s\mapsto v_s(1)$ is the solution of \eqref{2.2} with initial value
$v_0(1)=1$. Moreover,
\[
\lim_{t\to\infty}t^{-1/2}\int_0^t v_s(1)\,ds=0.
\]
\end{enumerate}
\end{thrm}

\begin{cor}
Suppose that $\beta=0$ and Condition \ref{condition absolute} holds. If, in addition,
$$
\int_0^1\frac{u}{\Psi(u)}\,du<\infty,
$$
the estimate in the critical case reduces to
\[
\|P_t(x,\cdot)-P_t(y,\cdot)\|_{\rm Var}
\leq\frac{C_1(1+|x-y|)}{\sqrt{t}}.
\]
\end{cor}

\proof  Note that
\[
\frac{d v_s(1)}{d s}=-\Psi(v_s(1)),\quad v_0(1)=1.
\]
By a change of variables, we have
\begin{align}\label{change of vari}
t=\int_{v_t(1)}^1\frac{du}{\Psi(u)},\qquad \int_0^t v_s(1)\,d s=\int_{v_t(1)}^1\frac{u}{\Psi(u)}\, du.
\end{align}
By \eqref{2.2}, \(s\mapsto v_s(1)\) is decreasing and $\lim_{s\to\infty}v_s(1)=0$.
This implies that
\begin{align}\label{change of vari of v}
\sup_{t}\int_0^t v_s(1)\,d s=\int_0^\infty v_s(1)\, ds=\int_0^1\frac{u}{\Psi(u)}\,du.
\end{align}
This proves the corollary.
\qed

The following examples give the
explicit upper bounds for some critical CBI processes.

\begin{exam}\label{stable sharp rate}
Let $1<\alpha<2$, $c>0$, and consider the critical CBI process with
\beqlb\label{stable mechanism}
\Psi(\lambda)=c\lambda^\alpha,\qquad \Phi(\lambda)=b\lambda,\qquad b>0.
\eeqlb
Then Condition \ref{condition absolute} holds.  Moreover,  there is a constant $C_2>0$ such that the following bound holds for all $x,y\in\mbb R_+$ and $t>0$,
\[
\|P_t(x,\cdot)-P_t(y,\cdot)\|_{\rm Var}\leq
C_2|x-y|t^{-1/(\alpha-1)}.
\]
\end{exam}

\proof
In the case of \(\Psi(\lambda)=c\lambda^\alpha\), one sees that
\[
\nu_\alpha(d\xi)=\frac{c\alpha(\alpha-1)}{\Gamma(2-\alpha)}\xi^{-1-\alpha}\,d\xi.
\]
This density is positive on every compact subinterval of $(0,\infty)$, so Condition \ref{condition absolute} holds. The cumulant integral is finite since
\[
\int_0^1\frac{u}{\Psi(u)}\,du
=\frac{1}{c}\int_0^1u^{1-\alpha}\,du
=\frac{1}{c(2-\alpha)}<\infty.
\]
On the other hand, the solution of \eqref{2.2} is
\begin{align}\label{stable cumulant}
v_t(\lambda)=\left(\lambda^{1-\alpha}
+c(\alpha-1)t\right)^{-1/(\alpha-1)},
\end{align}
and hence
\[
\bar v_t=\lim_{\lambda\to\infty}v_t(\lambda)
=\left(c(\alpha-1)t\right)^{-1/(\alpha-1)}.
\]
The explicit cumulant also verifies Grey's condition \eqref{t2.1}. Hence estimate \eqref{estimates1} applies, and
\[
\|P_t(x,\cdot)-P_t(y,\cdot)\|_{\rm Var}
\leq2(1-e^{-|x-y|\bar v_t})
\leq2|x-y|\left(c(\alpha-1)t\right)^{-1/(\alpha-1)}.
\]
This proves the asserted rate.\qed

\begin{exam}
Let \(b>0\) and consider a critical CBI process with mechanisms
\[
\Psi(\lambda)=\lambda\log(1+\lambda),
\qquad \Phi(\lambda)=b\lambda.
\]
Then there is a constant $C_3>0$ such that the following bound holds for all $x,y\in\mbb R_+$ and $t>1$,
\[
\|P_t(x,\cdot)-P_t(y,\cdot)\|_{\rm Var}
\leq
\frac{C_3}{\sqrt{t}}
\Big(1+|x-y|\log(1+t)\Big).
\]
\end{exam}

\proof The function
\(\Psi\) is a branching mechanism of the form
\[
\Psi(\lambda)=\int_0^\infty
\big(e^{-\lambda\xi}-1+\lambda\xi\big)\nu(d\xi),\quad \nu(d\xi)=\frac{(1+\xi)e^{-\xi}}{\xi^2}\,d\xi, \quad \lambda\ge0.
\]
The density of \(\nu\) is strictly positive and continuous on every compact interval of \((0,\infty)\). Hence Condition \ref{condition absolute} holds. However, Grey's condition fails.

For any $\varepsilon\in(0,1)$, there exists $\delta>0$ such that for any $u\in(0,\delta)$,
\[
1-\varepsilon\le \frac{u}{\log(1+u)}\le 1+\varepsilon.
\]
Note that $\lim_{t\to\infty}v_t(1)=0$. Then for $v_t(1)\in(0,\delta)$,
\[
(1-\varepsilon)\int_{v_t(1)}^\delta\frac{du}{u^2}\le
\int_{v_t(1)}^\delta \frac{du}{\Psi(u)}\le(1+\varepsilon)\int_{v_t(1)}^\delta\frac{du}{u^2}.
\]
Together with \eqref{change of vari}, this implies that
\[
(1-\varepsilon)\left(\frac{1}{v_t(1)}-\frac{1}{\delta}\right)\le t-\int_\delta^1\frac{du}{\Psi(u)}
\le (1+\varepsilon)\left(\frac{1}{v_t(1)}-\frac{1}{\delta}\right),
\]
and so
\[
\lim_{t\to\infty}tv_t(1)=1.
\]
By a similar argument, one can also obtain from the identity \eqref{change of vari} that
\[
\lim_{t\to\infty}\frac{\int_0^t v_s(1)ds}{\log(1/v_t(1))}=1.
\]
Thus
\[
\lim_{t\to\infty}\frac{1}{\log t}\int_0^t v_s(1)ds=1,
\]
which implies the desired result.\qed

We now focus on the exponential ergodicity. Write $\delta_x$ for the Dirac measure at $x$, and use $*$ for convolution.

\begin{con1}\label{condition basic}
There exist $x_0 \in \mathbb{R}_+$ and $\varepsilon > 0$ such that $B(x_0, \varepsilon) \subseteq (0, \infty)$ and
\begin{align}\label{expon cond}
K_{x_0, \varepsilon}:=\sup_{x}|x|^{-1}
\|\bar{\nu}_{x_0, \varepsilon}\ast\delta_x-\bar{\nu}_{x_0, \varepsilon}\|_{\rm Var}<\infty
\end{align}
where $\bar{\nu}_{x_0, \varepsilon}(\cdot)=\frac{\nu(\cdot\cap B(x_0, \varepsilon))}{\nu(B(x_0, \varepsilon))}$.
\end{con1}

\begin{thrm}\label{main theorem2}
Assume that $\beta<0$ and Condition \ref{condition basic} holds. Then there are constants $C_2,\lambda>0$ such that, for all $x,y\in\mbb R_+$ and $t>0$,
\[
\|P_t(x,\cdot)-P_t(y,\cdot)\|_{\rm Var}\leq C_2(1+|x-y|)e^{-\lambda t}.
\]
\end{thrm}

\begin{remark}
We make some comments on Theorem \ref{main theorem} and Theorem \ref{main theorem2}.

(a). Condition \ref{condition absolute} is inspired by Wang \cite{wang11} and formulates a similar hypothesis.  \cite{wang11} considered a $d$-dimensional Ornstein--Uhlenbeck process $\{Z_t:t\ge0\}$ determined by the following SDE:
\begin{align}\label{equation: OU}
dZ_t = AZ_t dt + BdL_t,
\end{align}
where $L=\{L_t:t\ge0\}$ is a L\'evy jump process with L\'evy measure $\mu$. By imposing an assumption on $\mu$ analogous to Condition \ref{condition absolute},  \cite{wang11} proved the successful coupling of the solution to \eqref{equation: OU} in the (sub)critical case (i.e., $\langle Ax, x \rangle \le 0$ for all $x\in\mathbb{R}^d$). Our approach for proving Theorem \ref{main theorem} builds on the coupling methods developed in \cite{wang11}. However, there are several essential differences between CBI processes and Ornstein--Uhlenbeck processes. For instance, Ornstein--Uhlenbeck processes require the driving noise to be a L\'evy process with non-degenerate coefficients, whereas the coefficients in the branching setting are degenerate on $\mathbb{R}_+$. This degeneracy presents challenges in directly applying the coupling techniques from \cite{wang11}. To address this, we first derive cluster representations of CBI  processes, decomposing their sample paths into three independent components:
$$
X^x_t=Y^x_t+\hat{X}_t+\tilde{X}_t,\quad \forall t\geq0,
$$
as detailed in \eqref{CBI cluster}. This decomposition reflects a branching structure where each cluster process is initiated by a ``mother jump'' followed by ``child jumps'', a structure widely used in mathematical finance; see, e.g., Jiao et al. \cite{JMSZ21}. Second, leveraging the independence of these components, we analyze ergodicity under a probability measure conditioned on $\{\hat{X}_t : t \geq 0\}$, which can be interpreted as a time-inhomogeneous immigration. Finally, we establish the main result through further estimates of negative moments of CBI processes.

(b). Condition \ref{condition basic} is inspired by the framework in Wang \cite[Theorem 1, formula (7)]{WangJ12}, which studied the SDE \eqref{equation: OU} with $B$ taken as the identity matrix. Given $\sigma$-finite measures $\nu_1$ and $\nu_2$ on $\mbb{R}_+$, we write $\nu_1\land \nu_2 := \nu_1 - (\nu_1-\nu_2)^+ = \nu_2 - (\nu_2-\nu_1)^+$, where the subscript ``$+$'' stands for the upper variation of the signed measure in its Jordan decomposition. According to \cite[Remark 1]{WangJ12}, \eqref{expon cond} implies
$$
\inf_{|x|\le x_0}\bar{\nu}_{x_0, \varepsilon}\wedge(\delta_x\ast\bar{\nu}_{x_0, \varepsilon})(\mathbb{R}_+)>0.
$$
This lower bound guarantees a positive chance of simultaneous jumps in the two marginals, as in Schilling and Wang \cite{SW11}. The cluster decomposition of CBI processes is likewise central to the proof of exponential ergodicity.
We also point out that Theorem \ref{main theorem2} remains valid when \eqref{expon cond} is replaced by Condition \ref{condition absolute}, with the additional requirement that
\[
\limsup_{r\rightarrow0}r^{-1}\sup_{|x|\leq r}\int_{B(x_0,\varepsilon)}|q(\xi)-q(x+\xi)|d\xi<\infty.
\]
\end{remark}

The strong Feller property  of one-dimensional CBI processes was established in
Li and Ma \cite{LM15} under Condition \eqref{t2.1}. The analytic properties of a finite-dimensional stable jump-type CBI
process were studied by Friesen and Jin \cite{FJ20}, who proved that the transition kernel of the process satisfies an a priori bound in a weighted anisotropic Besov norm. From this regularity they
deduced the strong Feller property and proved in the subcritical case the exponential ergodicity
in the total variation distance. To prove the strong Feller property in our setting, let us consider the following condition.

\bgcondition\label{condition strong feller}
Condition \ref{condition basic} holds for a sequence $\{\varepsilon_n\}\subset(0,1]$ such that $\nu(B(x_0,\varepsilon_n))\to\infty$ as $n\rightarrow\infty$.
\edcondition

\begin{thrm}\label{main theorem3}
If $\beta<0$ and Condition \ref{condition strong feller} holds, then $(P_t)_{t\ge0}$ is strong Feller.
\end{thrm}

\section{The cluster decomposition of CBI processes}

Either Condition \ref{condition absolute} or Condition \ref{condition basic} implies $x_0-\varepsilon>0$, so $B(x_0,\varepsilon)$ is bounded away from zero and $\nu(B(x_0,\varepsilon))<\infty$. We next construct a jump-cluster decomposition of the CBI process. Each cluster begins with a ``mother jump'' and evolves thereafter as a CB process.

We first remove jumps whose sizes lie in $B(x_0,\varepsilon)$ and use the resulting process as the base component:
\begin{align}\label{CBI_hat}
\hat{X}_t=\int_0^t (b+\hat{\beta}\hat{X}_s)ds+\int_0^t\int_0^{\hat{X}_{s-}}\int_{B^c(x_0,\varepsilon)}\xi\,\tilde{N}(ds,du,d\xi),\quad t\geq0,
\end{align}
where $\hat{\beta}=\beta-\int_{B(x_0,\varepsilon)}\xi\,\nu(d\xi)<0$.
It is straightforward to see that the process $\{\hat X_t:t\geq0\}$ is a CBI process with mechanisms $(\hat\Psi,\Phi)$, where
\begin{align}\label{mechanism of hat X}
\hat{\Psi}(\lambda) = -\hat{\beta}\lambda+ \int_{B^c(x_0,\varepsilon)}(e^{-\lambda \xi}-1+\lambda\xi)\nu(d\xi),\quad \lambda\geq0.
\end{align}
Let $N_0(ds,du,d\xi)$ be an independent copy of $N(ds,du,d\xi)$. Then we define a flow of CB processes with branching mechanism $\Psi$, as the initial cluster process, by
\[
Y^x_t=x+\beta \int_0^t Y^x_sds+\int_{0}^t\int_{0}^{Y^x_{s-}}\int_0^\infty \xi\,\tilde{N}_0(ds,du,d\xi), \quad t\geq0,\;x\geq0.
\]
Theorem \ref{branching property} shows that, for $x\ge y\ge0$, the difference $Y^x-Y^y$ is a CB process started from $x-y$ and is independent of $Y^y$.

Let $\mathbb{D}[0,\infty)$ be the space of c\`{a}dl\`{a}g paths $t\mapsto w(t)$ from
$[0,\infty)$ to $[0,\infty)$. On this space, define the $\sigma$-algebra  as $\mathcal{A} =\sigma\{w(s): 0\leq s\leq\infty\}$ and $\mathcal{A}_t =\sigma\{w(s): 0\leq s\leq t\} $ for $t\geq0$. Let $\mathbb{Q}_x(dw)$ denote the law of the process $\{Y^x_t:t\geq0\}$ with initial value $x$ on
$({\mathbb D}[0,\infty), \mathcal{A})$. To introduce a sequence of cluster processes, we need to define the cluster measure
$\rho(dw)$ on $({\mathbb D}[0,\infty), \mathcal{A})$ by
 \begin{align}\label{excursion}
\rho(dw)=\int_{B(x_0,\varepsilon)} \nu(dx)\mathbb{Q}_x(dw).
 \end{align}
Then we have the following cluster decomposition for CBI processes with mechanisms ($\Psi,\Phi$) .

\begin{prop}\label{cluster decomposition}
Let $M(ds,du,dw)$ be a Poisson random measure on $(0,\infty)^2\times {\mathbb D}[0,\infty)$ with intensity measure
$dsdu\rho(dw)$ by \eqref{excursion}. Suppose that $M$, $N$ and $N_0$ are independent of each other.
For $t\geq0$ define
\begin{align}\label{CBI cluster}
X^x_t=Y^x_t+\hat{X}_t+\int_0^t\int_0^{\hat{X}_{s-}}\int_{{\mathbb D}} w(t-s) M(ds,du,dw).
\end{align}
Then $\{X^x_t: t\geq0\}$ is a CBI process with initial value $x$ and mechanisms given by \eqref{mechanism}.
\end{prop}

\proof
Define ${\mathcal{F}^Y_t}=\sigma\{Y^x_s: 0\leq s\leq t\}$, ${\hat{\mathcal{F}}_t}=\sigma\{\hat{X}_s: 0\leq s\leq t\}$, and
$\mathcal{F}^M_t=\sigma\{M ((0, s]\times U\times A) : U\in\mathcal{B}(\mathbb{R}_+),\; A \in\mathcal{A}_{t-s},\;0\leq s\leq t\}$.
Let $\mathcal{G}_t=
\mathcal{F}^M_t\vee\hat{\mathcal{F}}_t\vee\mathcal{F}^Y_t$. Note that $N$, $N_0$ and $M$ are independent of each other.
Then $\{Y^x_t: t\geq0\}$, $\{\hat{X}_t: t\geq0\}$ and $M$ are independent of each other.  For $0\leq r\leq t$ and $\lambda\geq0$, we have
\begin{align}\label{trans}
\mathbb{E}[e^{-\lambda X^x_t}|\,\mathcal{G}_r]&=\mathbb{E}[e^{-\lambda Y^x_t}|\,\mathcal{F}^Y_r]\mathbb{E}\Big[e^{-\lambda\hat{X}_t- \lambda\int_0^t
\int_0^{\hat{X}_{s-}}\int_{\mathbb{D}} w(t-s) M(ds,du,dw)}|\; \mathcal{F}^M_r\cup\hat{\mathcal{F}}_r\Big]\cr
&= e^{-Y_r^x v_{t-r}(\lambda)}\cdot I_1,
\end{align}
where $v_t(\lambda)$ is given by \eqref{2.2}. Note that
\begin{align}\label{I_1}
I_1&= \mathbb{E} \bigg[\mathbb{E}\big[e^{-\lambda\hat{X}_t-\int_0^t
\int_0^{\hat{X}_{s-}}\int_{\mathbb{D}} \lambda w(t-s) M(ds,du,dw)}\big|\; \mathcal{F}^M_r\cup\hat{\mathcal{F}}_t\big]
\bigg| \mathcal{F}^M_r\cup\hat{\mathcal{F}}_r\bigg]\cr
&=\mathbb{E}\bigg[e^{-\lambda\hat{X}_t}\mathbb{E}\big[e^{-\int_0^t\int_0^{\hat{X}_{s-}}\int_{\mathbb{D}}  \lambda w(t-s) M(ds,du,dw)}\,\big|\mathcal{F}^M_r\cup\hat{\mathcal{F}}_t\big]\bigg| \mathcal{F}^M_r\cup\hat{\mathcal{F}}_r\bigg]\cr
&=\mathbb{E}\big[e^{-\lambda\hat{X}_t}\cdot I_2\,\big| \mathcal{F}^M_r\cup\hat{\mathcal{F}}_r\big].
\end{align}
Now we turn to $I_2$. As mentioned above, $\hat{X}$ and $M$ are independent of each other. Furthermore, conditional on $\hat{X}$, $1_{[0,\hat{X}_{s-}]}(u)M(ds,du,dw)$ is a time-inhomogeneous Poisson point process and thus has independent increments. Then
\begin{align}\label{I_2}
I_2&=\mathbb{E}\big[e^{-(\int_0^r+\int_r^t)\int_0^{\hat{X}_{s-}}\int_{\mathbb{D}} \lambda w(t-s) M(ds,du,dw)}\,\big|\mathcal{F}^M_r\cup\hat{\mathcal{F}}_t\big]\cr
&=\mathbb{E}\big[e^{-\int_0^r\int_0^{\hat{X}_{s-}}\int_{\mathbb{D}}\lambda w(t-s)M(ds,du,dw)}\big|\mathcal{F}^M_r
\cup\hat{\mathcal{F}}_t\big]\mathbb{E} \big[e^{-\int_r^t\int_0^{\hat{X}_{s-}}\int_{\mathbb{D}}\lambda w(t-s)M(ds,du,dw)}\big|\hat{\mathcal{F}}_t\big]\cr
&= e^{-\int_0^r\int_0^{\hat{X}_{s-}}\int_{\mathbb{D}}w(r-s)v_{t-r}(\lambda)M(ds,du,dw)}
e^{\int_r^t\hat{X}_sds\int_{B(x_0,\varepsilon)}(e^{-xv_{t-s}(\lambda)}-1)\nu(dx)},
\end{align}
where the final equality follows from the branching property and the exponential formula. Applying It\^o's formula to the truncated process, we have
\[
e^{-v_{t-u}(\lambda)\hat{X}_u-\int_0^u \hat{X}_sds\int_{B(x_0,\varepsilon)}(1-e^{-xv_{t-s}(\lambda)})\nu(d x)+b\int_0^u
v_{t-s}(\lambda)ds},\quad 0\leq u\leq t,
\]
is a $\hat{\mathcal{F}}_u$-martingale. Note that $v_0(\lambda)=\lambda$.  Then
\[
\begin{aligned}
&\mathbb{E}\Big[e^{-\lambda \hat{X}_t-\int_0^t\hat{X}_sds\int_{B(x_0,\varepsilon)}(1-e^{-xv_{t-s}(\lambda)})\nu(dx)+b\int_0^tv_{t-s}(\lambda)ds}\big|\hat{\mathcal{F}}_r\Big]\\
&= e^{-v_{t-r}(\lambda)\hat{X}_r-\int_0^r \hat{X}_sds\int_{B(x_0,\varepsilon)}(1-e^{-xv_{t-s}(\lambda)})\nu(dx)+b\int_0^r v_{t-s}(\lambda)ds},
\end{aligned}
\]
which implies that
\begin{align}\label{mart}
\mathbb{E}\Big[e^{-\lambda \hat{X}_t-\int_r^t\hat{X}_sds\int_{B(x_0,\varepsilon)}(1-e^{-xv_{t-s}(\lambda)})\nu(dx)}\big|\hat{\mathcal{F}}_r\Big]=e^{-v_{t-r}(\lambda)\hat{X}_r-b\int_r^tv_{t-s}(\lambda)ds}.
\end{align}
By \eqref{trans}, \eqref{I_1} and \eqref{I_2}, we have
\[
\begin{aligned}
&\mathbb{E}[e^{-\lambda X^x_t}|\,\mathcal{G}_r]\cr
&= e^{-Y_r^x v_{t-r}(\lambda)}e^{-\int_0^r\int_0^{\hat{X}_{s-}}\int_{\mathbb{D}}w(r-s)v_{t-r}(\lambda)M(ds,du,dw)}\cr &\qquad\qquad\cdot\mathbb{E}\Big[e^{-\lambda \hat{X}_t-\int_r^t\hat{X}_sds\int_{B(x_0,\varepsilon)}(1-e^{-xv_{t-s}(\lambda)})\nu(dx)}\big|\mathcal{F}^M_r
\cup\hat{\mathcal{F}}_r\Big]\cr
&= e^{-Y_r^x v_{t-r}(\lambda)}e^{-\int_0^r\int_0^{\hat{X}_{s-}}\int_{\mathbb{D}}w(r-s)v_{t-r}(\lambda)M(ds,du,dw)}\cr &\qquad\qquad\cdot\mathbb{E}\Big[e^{-\lambda \hat{X}_t-\int_r^t\hat{X}_sds\int_{B(x_0,\varepsilon)}(1-e^{-xv_{t-s}(\lambda)})\nu(dx)}
\big|\hat{\mathcal{F}}_r\Big]\cr
&= e^{-v_{t-r}(\lambda)X_r^x-b\int_r^tv_{t-s}(\lambda)ds},
\end{aligned}
\]
which shows that $\{X^x_t: t\geq0\}$ is a CBI process with mechanisms $(\Psi, \Phi)$. The second equality follows from the fact that
 $\hat{X}$ and $M$ are independent of each other and the last equality follows from \eqref{CBI cluster} and \eqref{mart}.
\qed

\section{Proof of main results}

\subsection{Negative moments of CBI processes}

This subsection establishes bounds for two negative moments of CBI occupation integrals. We begin with the following lemma.
\blemma\label{Sch02}(Sch\"urger \cite[Theorem 1.1]{Sch})
Let $X\ge0$ be a random variable with Laplace transform $\mathcal{L}_X$. Then
\[
\mbb{E}\Big[X^{-r}\Big]=\frac{1}{r\Gamma{(r)}}\int_0^\infty \mathcal{L}_X(t^{1/r}) dt,\quad r>0,
\]
where $\Gamma$ denotes the Gamma function.
\elemma

The main estimate is the following proposition.

\begin{prop}\label{moment of 1/Y}
Consider a CBI process $\{X_t:t\geq0\}$ starting from $x\geq0$ with mechanisms ($\Psi,\Phi$). Suppose that $\beta<0$ or $\nu\neq0$. Then there exist $K_0,K_1>0$ independent of $x$, such that
\[
\mbb{E}_x\Big[\Big(\int_0^t X_s ds\Big)^{-\frac{1}{2}}\Big]\leq \frac{K_0}{t\wedge\sqrt{t}},\qquad \mbb{E}_x\Big[\Big(\int_0^t X_s ds\Big)^{-1}\Big]\leq \frac{K_1}{t^2\wedge t},\qquad t>0,
\]
where $\mbb{E}_x(\cdot):=\mbb{E}(\cdot|X_0=x)$.
\end{prop}
\proof  By Filipovi\'c \cite[Theorem 5.3]{Fil01}, the Laplace transform of the occupation integral is
\begin{align}\label{laplace of integ}
\mbb{E}_x\left[\exp\left\{-\lambda\int_0^t X_s ds\right\}\right]
=\exp\left(-xV_t(\lambda)-b\int_0^t V_s(\lambda) ds\right),
\end{align}
where $t\mapsto V_t(\lambda)$ is the unique solution of
$$
\frac{\partial V_t(\lambda)}{\partial t}=-\Psi(V_t(\lambda))+\lambda,\qquad V_0(\lambda)=0;
$$
see also Li \cite[Chapter 4]{Li20}. For fixed $\delta>0$,
\[
\begin{aligned}
\Psi(u)&=\left(-\beta+\int_\delta^{\infty}\xi \nu(d\xi)\right)u
+\int_0^\delta(e^{-u\xi}-1+u\xi) \nu(d\xi)+\int_\delta^{\infty}(e^{-u\xi}-1) \nu(\mrm{d}\xi)\\
&\leq\left(-\beta+\int_\delta^{\infty}\xi \nu(d\xi)\right)u+\frac{1}{2}\int_0^\delta\xi^2 \nu(d\xi)u^2\leq\hat{\gamma}u+\hat{c}u^2,
\end{aligned}
\]
where
\[
\hat{\gamma}=1-\beta+\int_\delta^\infty\xi \nu(d\xi)>1,\qquad
\hat{c}=1+\frac{1}{2}\int_0^\delta\xi^2 \nu(d\xi)>1.
\]
Therefore,
\[
\frac{\partial V_t(\lambda)}{\partial t}\geq-\hat{\gamma}V_t(\lambda)-\hat{c}V_t^2(\lambda).
\]
Comparison with the associated Riccati equation yields
\[
\begin{aligned}
V_t(\lambda)&\geq\frac{2\lambda(e^{\sqrt{\hat{\gamma}^2+4\hat{c}\lambda}t}-1)}
{e^{\sqrt{\hat{\gamma}^2+4\hat{c}\lambda}t}(\sqrt{\hat{\gamma}^2+4\hat{c}\lambda}+\hat{\gamma})
+\sqrt{\hat{\gamma}^2+4\hat{c}\lambda}-\hat{\gamma}}\\
&=\frac{2\lambda(1-e^{-\sqrt{\hat{\gamma}^2+4\hat{c}\lambda}t})}
{\sqrt{\hat{\gamma}^2+4\hat{c}\lambda}+\hat{\gamma}
+(\sqrt{\hat{\gamma}^2+4\hat{c}\lambda}-\hat{\gamma})e^{-\sqrt{\hat{\gamma}^2+4\hat{c}\lambda}t}}\\
&\geq\frac{\lambda(1-e^{-\sqrt{\hat{\gamma}^2+4\hat{c}\lambda}t})}{\sqrt{\hat{\gamma}^2+4\hat{c}\lambda}}.
\end{aligned}
\]
Write $D_\lambda=(\hat{\gamma}^2+4\hat{c}\lambda)^{1/2}$.  An application of integration gives
\begin{align}\label{integrated Riccati lower}
\int_0^tV_s(\lambda)\,ds
\geq\frac{\lambda}{D_\lambda}
\left(t-\frac{1-e^{-D_\lambda t}}{D_\lambda}\right)
\geq e^{-1}\frac{\lambda t}{D_\lambda}(1\wedge D_\lambda t),
\end{align}
where we have used the elementary inequality
$1-(1-e^{-z})/z\geq e^{-1}(1\wedge z)$ for $z>0$.

For $p\in\{1,2\}$, combining \eqref{laplace of integ}, \eqref{integrated Riccati lower} with Lemma \ref{Sch02} gives
\begin{align}\label{Laplace integral bound}
\mbb E_x\left[\left(\int_0^t X_s ds\right)^{-\frac{1}{p}}\right]&=
\frac{p}{\Gamma(1/p)}
\int_0^\infty
\mbb E_x\left[
\exp\left(-\lambda^p\int_0^tX_s\,ds\right)\right]d\lambda\cr
&\leq\frac{p}{\Gamma(1/p)}\int_0^\infty
\exp\left[-\frac{b\lambda^pt}{eD_{\lambda^p}}
\left(1\wedge D_{\lambda^p}t\right)\right]d\lambda.
\end{align}
If $t>1$, then $D_{\lambda^p}t>1$. It follows that
\begin{align}\label{est of t>1}
\int_0^\infty
\exp\left[-\frac{b\lambda^pt}{eD_{\lambda^p}}
\left(1\wedge D_{\lambda^p}t\right)\right]d\lambda
&=\int_0^\infty \exp\left(-\frac{b\lambda^p t}{eD_{\lambda^p}}\right)d\lambda\cr
&\le\int_0^1 e^{-c_1\lambda^pt}d\lambda
+\int_1^\infty e^{-c_2\lambda^{p/2}t}d\lambda\le c_3t^{-1/p}
\end{align}
for some positive constants
$c_i (i=1,2,3)$. We now suppose that $0<t\le 1$. By the definition of $D_\lambda$, there are
constants $0<c_4<1 <c_5<\infty$ such that
\begin{align}\label{ineq:3.4}
c_4(1+\lambda^{p/2})\leq D_{\lambda^p}
\leq c_5(1+\lambda^{p/2}),\qquad \lambda\geq0.
\end{align}
 On the interval $\lambda\in[0,t^{-2/p}]$, we have from \eqref{ineq:3.4} that
\[
D_{\lambda^p}\le c_5(1+\lambda^{p/2})
\le c_5\left(1+\frac1t\right)\le \frac{2c_5}{t}, \qquad
1\wedge D_{\lambda^p}t\ge\frac{D_{\lambda^p}t}{2c_5},
\]
which implies that
\begin{align}\label{ineq:3.5}
\frac{\lambda^p t}{D_{\lambda^p}}(1\wedge D_{\lambda^p}t)\ge
\frac{\lambda^pt^2}{2c_5}.
\end{align}
On the other hand, on the interval $\lambda\in(t^{-2/p},\infty)$, we have from \eqref{ineq:3.4} that
\[
1\wedge D_{\lambda^p}t\ge 1\wedge \left[c_4\left(1+\lambda^{p/2}\right) t\right]\ge 1\wedge c_4=c_4,
\]
which implies
\begin{align}\label{ineq:3.6}
\frac{\lambda^p t}{D_{\lambda^p}}(1\wedge D_{\lambda^p}t)\ge
\frac{c_4\lambda^pt}{c_5(1+\lambda^{p/2})}\ge\frac{c_4\lambda^{p/2}t}{2c_5}.
\end{align}
Then, by \eqref{ineq:3.5} and \eqref{ineq:3.6}, for $0<t\le1$,
\begin{align}\label{est of t<1}
\int_0^\infty
\exp\left[-\frac{b\lambda^pt}{eD_{\lambda^p}}
\left(1\wedge D_{\lambda^p}t\right)\right]d\lambda
&\le\int_0^{t^{-2/p}} e^{-c_6\lambda^pt^2}d\lambda
+\int_{t^{-2/p}}^\infty e^{-c_7\lambda^{p/2}t}d\lambda\le c_8t^{-2/p}
\end{align}
for some positive constants
$c_i (i=6,7,8)$.

In conclusion, by \eqref{est of t>1} and \eqref{est of t<1},
\[
\begin{aligned}
\int_0^\infty
\exp\left\{-\frac{b\lambda^pt}{eD_{\lambda^p}}
(1\wedge D_{\lambda^p}t)\right\}d\lambda
\leq C_p
\begin{cases}
t^{-2/p},&0<t\leq1,\\
t^{-1/p},&t>1
\end{cases}
\end{aligned}
\]
for some $C_p>0$. Thanks to \eqref{Laplace integral bound}, we complete the proof.
\qed

\subsection{Proof of Theorem \ref{main theorem}}

We need a crucial lemma.

\blemma\label{Palm} (Palm formula)
Let $E_p$ denote the space of point measures on $E$, and let $G: E\times E_p\rightarrow\mathbb{R}_+$ be a measurable functional. If $N$ is a Poisson point measure with intensity $n$, then we have
\[
 \mathbb{E}\left(\int_E f(x)G(x,N)N(dx)\right)=\int_E\mathbb{E}(G(x,\delta_x+N))
 f(x)n(dx),\qquad f: E\to\mathbb R_+.
\]
\elemma

We refer to Bertoin \cite[Lemma 2.3]{B06} for a proof.

\medskip

\noindent{\it Proof of Theorem \ref{main theorem}.} We divide the proof into three steps.

{\bf Step 1}. From now on, we write $B(\varepsilon)=B(x_0,\varepsilon)$ for simplicity. Fix $T>0$. Define
\[
\hat{A}_T:=\int_0^T\hat{X}_{s-}\,ds,\qquad \hat{\mathbb{P}}(\cdot):=\mathbb{P}(\cdot\mid\sigma\{\hat{X}_t:t\geq0\})
\]
and let $\hat{\mbb{E}}$ denote the associated expectation. Under $\hat{\mathbb{P}}$, let $(\tau,\zeta,\xi)$ be a random vector on $\mathcal E:=(0,\infty)^2\times{\mathbb D}[0,\infty)$ with distribution
$$
\frac{\mbf{1}_{(0,T]}(s)\mbf{1}_{(0,\hat{X}_{s-}]}(u)\mbf{1}_{B(\varepsilon/2)}(w(0))dsdu\rho(dw)}
{\nu(B(\varepsilon/2))\hat{A}_T}
$$
such that {$(\tau,\zeta,\xi), Y^x$ and $M$ are independent under $\hat{\mbb{P}}$}. Let
$D_p$ be the space of point measures on $\mathcal E$ and let $F: D_p\rightarrow\mbb{R}_+$ be a measurable functional. Then $M\in D_p$ and $M+\delta_{(\tau,\zeta,\xi)}\in D_p$. By Lemma \ref{Palm},
\begin{align}\label{palm appli}
&\hat{\mbb{E}}[F(M+\delta_{(\tau,\zeta,\xi)})]\cr
=&\hat{\mbb{E}}\Big[\int_{\mathcal E}
\frac{\mbf{1}_{(0,T]}(s)\mbf{1}_{(0,\hat{X}_{s-}]}(u)\mbf{1}_{B(\varepsilon/2)}(w(0))}
{\nu(B(\varepsilon/2))\hat{A}_T}\,M(ds,du,dw)F(M)\Big]\cr
=&\frac{\hat{\mbb{E}}\Big[\int_0^T\int_0^{\hat{X}_{s-}}\int_{{\mathbb D}[0,\infty)}
\mbf{1}_{B(\varepsilon/2)}(w(0))\,M(ds,du,dw)F(M)\Big]}
{\nu(B(\varepsilon/2))\hat{A}_T}.
\end{align}
For fixed $w'\in{\mathbb D}[0,\infty)$, we have
$$
\gamma\mapsto w'(T)+\hat{X}_T+\int_0^T\int_0^{\hat{X}_{s-}}\int_{{\mathbb D}[0,\infty)}w(T-s)\,\gamma(ds,du,dw),\quad \gamma\in D_p
$$
is a measurable functional from $D_p$ to $\mbb{R}_+$. For $f\in{\mathcal{B}_b(\mbb{R}_+)},$ write
\[
H(w',\hat{X},\gamma)=f\left(w'(T)+\hat{X}_T+\int_0^T\int_0^{\hat{X}_{s-}}\int_{{\mathbb D}[0,\infty)}w(T-s)\,\gamma(ds,du,dw)\right).
\]
Let
\[
\Theta_T^x(\gamma)=Y^x_T+\hat{X}_T
+\int_0^T\int_0^{\hat{X}_{s-}}\int_{{\mathbb D}[0,\infty)}
w(T-s)\,\gamma(ds,du,dw).
\]
Note that $\{Y^x_t:t\geq0\}$ is independent of $\{\hat{X}_t:t\geq0\}$, $M$ and $(\tau,\zeta,\xi)$. Then we have
\begin{align}\label{est in Step1}
&\hat{\mbb{E}}\left[f\big(\Theta_T^x(M+\delta_{(\tau,\zeta,\xi)})\big)\right]=\int_{{\mathbb D}[0,\infty)}\hat{\mbb{E}}[H(w',\hat{X},M+\delta_{(\tau,\zeta,\xi)})]\,\mathbb{Q}_x(dw')\cr
&=\int_{{\mathbb D}[0,\infty)}
\frac{\hat{\mbb{E}}\left[\int_0^T\int_0^{\hat{X}_{s-}}\int_{{\mathbb D}[0,\infty)}
\mbf{1}_{B(\varepsilon/2)}(w(0))\,M(ds,du,dw)H(w',\hat{X},M)\right]}
{\nu(B(\varepsilon/2))\hat{A}_T}\,\mathbb{Q}_x(dw')\cr
&=\frac{\hat{\mbb{E}}\left[\int_0^T\int_0^{\hat{X}_{s-}}\int_{{\mathbb D}[0,\infty)}
\mbf{1}_{B(\varepsilon/2)}(w(0))\,M(ds,du,dw)f\left(\Theta_T^x(M)\right)\right]}
{\nu(B(\varepsilon/2))\hat{A}_T}\cr
&=\frac{\hat{\mbb{E}}\left[\int_0^T\int_0^{\hat{X}_{s-}}\int_{{\mathbb D}[0,\infty)}
\mbf{1}_{B(\varepsilon/2)}(w(0))\,M(ds,du,dw)f(X^x_T)\right]}
{\nu(B(\varepsilon/2))\hat{A}_T}\cr
&=\frac{\hat{\mbb{E}}\left[f(X^x_T)\sum\limits_{i=1}^{J_T}\mbf{1}_{\{\eta_i\in B(\varepsilon/2)\}}\right]}
{\nu(B(\varepsilon/2))\hat{A}_T},
\end{align}
where the second equality follows from Lemma \ref{Palm},
\[
J_t:=
\int_0^t\int_0^{\hat X_{s-}}
\int_{\mathbb D[0,\infty)}
\mathbf 1_{\{w(0)\in B(\varepsilon)\}}
M(ds,du,dw),\qquad t\ge0
\]
and $\{\eta_i\}_{i\geq1}$ are i.i.d. random variables with distribution $\frac{\nu(\cdot\cap B(\varepsilon))}{\nu(B(\varepsilon))}$.
Note that under $\hat{\mathbb P}$, $\{J_t:t\ge0\}$ is a time-inhomogeneous Poisson process and $J_T$ has parameter $
\nu(B(\varepsilon))\hat{A}_T.$

{\bf Step 2}. Suppose that $x\geq y\geq0$. Under the common cluster construction \eqref{CBI cluster}, define
\[
\widetilde Y_t:=X_t^x-X_t^y=Y^x_t-Y^y_t.
\]
It follows from Theorem \ref{branching property} that $\tilde{Y}$ is also a CB process with probability distribution $\mathbb{Q}_{x-y}(dw)$ on ${\mathbb D}[0,\infty)$ and it is independent of $\{X^y_t:t\geq0\}$. For the random time $\tau$, we also define $\tilde{Y}_{\tau+\cdot}:=\{\tilde{Y}_{\tau+t}:t\geq0\}$. Then
\begin{align}\label{ineq:3.10}
&\hat{\mbb{E}}\left[f\left(Y^x_T+\hat{X}_T+\int_0^T\int_0^{\hat{X}_{s-}}\int_{{\mathbb D}[0,\infty)}w(T-s)\,(M+\delta_{(\tau,\zeta,\xi)})(ds,du,dw)\right)\right]\cr
=&\hat{\mbb{E}}\left[f\left(Y^y_T+\tilde{Y}_T+\hat{X}_T+\int_0^T\int_0^{\hat{X}_{s-}}\int_{{\mathbb D}[0,\infty)}w(T-s)\,(M+\delta_{(\tau,\zeta,\xi)})(ds,du,dw)\right)\right]\cr
=&\hat{\mbb{E}}\left[f\left(Y^y_T+\hat{X}_T+\int_0^T\int_0^{\hat{X}_{s-}}\int_{{\mathbb D}[0,\infty)}w(T-s)\,(M+\delta_{(\tau,\zeta,\xi+\tilde{Y}_{\tau+\cdot})})(ds,du,dw)\right)\right]\cr
=&\int_{{\mathbb D}[0,\infty)}\hat{\mbb{E}}\left[H\left(w',\hat{X},M+\delta_{(\tau,\zeta,\xi+\tilde{Y}_{\tau+\cdot})}\right)\right]\,\mathbb{Q}_y(dw'),
\end{align}
where the last equality uses the independence of $Y^y$ from $\tilde Y$, $(\tau,\zeta,\xi)$, and $M$. Fix $w'\in\mathbb D[0,\infty)$.
\begin{align}\label{ineq:3.11}
&\hat{\mbb{E}}\left[H\left(w',\hat{X},M+\delta_{(\tau,\zeta,\xi+\tilde{Y}_{\tau+\cdot})}\right)\right]\cr
=&\int_0^\infty\int_0^\infty\int_{{\mathbb D}[0,\infty)}\int_{{\mathbb D}[0,\infty)}
\hat{\mbb{E}}\left[H\left(w',\hat{X},M+\delta_{(s,u,w(\cdot)+w''(s+\cdot))}\right)\right]\cr
&\qqquad\quad\quad\times
\frac{\mbf{1}_{\{0\leq s\leq T\}}\mbf{1}_{\{0\leq u\leq \hat{X}_{s-}\}}\mbf{1}_{\{w(0)\in B(\varepsilon/2)\}}}{\nu(B(\varepsilon/2))\hat{A}_T}\,dsdu\rho(dw)\mathbb{Q}_{x-y}(dw'')\cr
=&\int_0^\infty\int_0^\infty\int_{{\mathbb D}[0,\infty)}
\hat{\mbb{E}}\left[H\left(w',\hat{X},M+\delta_{(s,u,w^*(\cdot))}\right)\right]\cr
&\qquad\quad\times
\frac{\mbf{1}_{\{0\leq s\leq T\}}\mbf{1}_{\{0\leq u\leq \hat{X}_{s-}\}}\mbf{1}_{\{r-z\in B(\varepsilon/2)\}}}{\nu(B(\varepsilon/2))\hat{A}_T}\,dsdu
\int_0^\infty\int_0^\infty\mathbb{Q}_{r}(dw^*)Q_s(x-y,dz)\nu(d(r-z))\cr
:=&I_1(w')+I_2(w'),
\end{align}
where
\[
\begin{aligned}
&I_1(w')=\int_0^\infty\int_0^\infty\int_{{\mathbb D}[0,\infty)}
\hat{\mbb{E}}\left[H\left(w',\hat{X},M+\delta_{(s,u,w^*(\cdot))}\right)\right]\cr
&\qquad\times
\frac{\mbf{1}_{\{0\leq s\leq T\}}\mbf{1}_{\{0\leq u\leq \hat{X}_{s-}\}}\mbf{1}_{\{r-z\in B(\varepsilon/2)\}}}{\nu(B(\varepsilon/2))\hat{A}_T}\,dsdu
\int_0^{\varepsilon/2}Q_s(x-y,dz)\int_0^\infty\mathbb{Q}_{r}(dw^*)\nu(d(r-z))\cr
&I_2(w')=\int_0^\infty\int_0^\infty\int_{{\mathbb D}[0,\infty)}
\hat{\mbb{E}}\left[H\left(w',\hat{X},M+\delta_{(s,u,w^*(\cdot))}\right)\right]\cr
&\qquad\times
\frac{\mbf{1}_{\{0\leq s\leq T\}}\mbf{1}_{\{0\leq u\leq \hat{X}_{s-}\}}\mbf{1}_{\{r-z\in B(\varepsilon/2)\}}}{\nu(B(\varepsilon/2))\hat{A}_T}\,dsdu
\int_{\varepsilon/2}^{\infty}Q_s(x-y,dz)\int_0^\infty\mathbb{Q}_{r}(dw^*)\nu(d(r-z)),
\end{aligned}
\]
and $w^*:=(w''(s+\cdot)+w(\cdot))$ denotes the path of a CB process with initial value $w''(s)+w(0)$ for any $s\geq0$. The second equality follows from the statement that for $u,v\geq0$, $\{Y^u_t+Y^v_t:t\geq0\}$ has the same distribution as $\{Y^{u+v}_t:t\geq0\}$, and $Q_s(x-y,\cdot)$ represents the transition semigroup of a CB process with initial value $x-y$.

Recall that $B(\varepsilon)$ is bounded away from 0. If $z\in(0,\varepsilon/2)$ and
$r-z\in B(\varepsilon/2)$, then $r\in B(\varepsilon)$. Using \(\nu(dr)=q(r)\,dr\) on \(B(\varepsilon)\), we have
\[
\int_{B(\varepsilon)}
\frac{q(r-z)}{q(r)}
\mathbf 1_{\{r-z\in B(\varepsilon/2)\}}\nu(dr)
=\nu(B(\varepsilon/2)),\quad z\in(0,\varepsilon/2).
\]
Then, by Lemma \ref{Palm} and Condition \ref{condition absolute}, we arrive at
\begin{align}\label{ineq:3.12}
I_1(w')
&=\int_0^\infty\int_0^\infty\int_{{\mathbb D}[0,\infty)}
\hat{\mbb{E}}\left[
H\left(w',\hat{X},M+\delta_{(s,u,w^*(\cdot))}\right)\right]\cr
&\quad\times
\frac{\mbf{1}_{\{0\leq s\leq T\}}
\mbf{1}_{\{0\leq u\leq \hat{X}_{s-}\}}
\mbf{1}_{\{r-z\in B(\varepsilon/2)\}}}
{\nu(B(\varepsilon/2))\hat{A}_T}\,dsdu\cr
&\quad\times
\int_0^{\varepsilon/2}Q_s(x-y,dz)
\int_0^\infty\mathbb{Q}_{r}(dw^*)\frac{q(r-z)}{q(r)}\nu(dr)\cr
&=\int_0^\infty\int_0^\infty\int_{{\mathbb D}[0,\infty)}
\hat{\mbb{E}}\left[
H\left(w',\hat{X},M+\delta_{(s,u,w^*(\cdot))}\right)\right]\cr
&\quad\times
\frac{\mbf{1}_{\{0\leq s\leq T\}}
\mbf{1}_{\{0\leq u\leq \hat{X}_{s-}\}}
\mbf{1}_{\{w^*(0)-z\in B(\varepsilon/2)\}}}
{\nu(B(\varepsilon/2))\hat{A}_T}\,dsdu\cr
&\quad\times
\int_0^{\varepsilon/2}Q_s(x-y,dz)
\int_0^\infty\mathbb{Q}_{r}(dw^*)
\frac{q(w^*(0)-z)}{q(w^*(0))}\nu(dr)\cr
&=\hat{\mbb{E}}\left[
H(w',\hat{X},M)
\int_0^T\int_0^{\hat{X}_{s-}}\int_{{\mathbb D}[0,\infty)}
\int_0^{\varepsilon/2}\frac{q(w^*(0)-z)}{q(w^*(0))}\right.\cr
&\left.\qquad\times
\frac{\mbf{1}_{\{w^*(0)-z\in B(\varepsilon/2)\}}}
{\nu(B(\varepsilon/2))\hat{A}_T}
\,Q_s(x-y,dz)M(ds,du,dw^*)\right].
\end{align}
Combining \eqref{ineq:3.10}, \eqref{ineq:3.11} and \eqref{ineq:3.12}, we have
\begin{align}\label{est in step2}
&\hat{\mbb{E}}\left[
f\left(Y^x_T+\hat{X}_T+
\int_0^T\int_0^{\hat{X}_{s-}}\int_{{\mathbb D}[0,\infty)}
w(T-s)\,(M+\delta_{(\tau,\zeta,\xi)})(ds,du,dw)\right)\right]\cr
=&\int_{{\mathbb D}[0,\infty)}(I_1(w')+I_2(w'))\,\mathbb{Q}_y(dw')\cr
=&\hat{\mbb{E}}\left[f(X^y_T)
\int_0^T\int_0^{\hat{X}_{s-}}\int_{{\mathbb D}[0,\infty)}
\int_0^{\varepsilon/2}\frac{q(w^*(0)-z)}{q(w^*(0))}\right.\cr
&\left.\qquad\qquad\times
\frac{\mbf{1}_{\{w^*(0)-z\in B(\varepsilon/2)\}}}
{\nu(B(\varepsilon/2))\hat{A}_T}
\,Q_s(x-y,dz)M(ds,du,dw^*)\right]+\int_{{\mathbb D}[0,\infty)}I_2(w')\,\mathbb{Q}_y(dw')\cr
=&
\frac{\hat{\mbb{E}}\left[f(X^y_T)\sum_{i=1}^{J_T}
\left(\int_0^{\varepsilon/2}\frac{q(\eta_i-z)}{q(\eta_i)}
\mbf{1}_{\{\eta_i-z\in B(\varepsilon/2)\}}
\,Q_{\tau_i}(x-y,dz)\right)\right]}
{\nu(B(\varepsilon/2))\hat{A}_T}
+\int_{{\mathbb D}[0,\infty)}I_2(w')\,\mathbb{Q}_y(dw'),
\end{align}
where $\tau_i$ is the $i$-th jump time of $\{J_t:t\in[0,T]\}$ for $i\geq1$.

{\bf Step 3}.
For $i\geq1$, define
\[
\bar{\varrho}_i:=
\frac{\nu(B(\varepsilon))}{\nu(B(\varepsilon/2))}
\mbf{1}_{\{\eta_i\in B(\varepsilon/2)\}},\qquad \tilde{\varrho}_i:=\int_{\varepsilon/2}^\infty Q_{\tau_i}(x-y,dz)
\]
and
\[
\varrho_i:=
\frac{\nu(B(\varepsilon))}{\nu(B(\varepsilon/2))}
\int_0^{\varepsilon/2}\frac{q(\eta_i-z)}{q(\eta_i)}
\mbf{1}_{\{\eta_i-z\in B(\varepsilon/2)\}}Q_{\tau_i}(x-y,dz)
+\tilde{\varrho}_i,
\]
where $\{\eta_i\}_{i\ge1}$ and $\{\tau_i\}_{i\ge1}$ are defined in Step~1 and Step~2, respectively. Let
\[
\bar{S}_T:=\frac{1}{\nu(B(\varepsilon))\hat{A}_T}\sum_{i=1}^{J_T}\bar{\varrho}_i,\qquad S_T:=\frac{1}{\nu(B(\varepsilon))\hat{A}_T}\sum_{i=1}^{J_T}\varrho_i,\qquad
\tilde{S}_T=\frac{1}{\nu(B(\varepsilon))\hat{A}_T}\sum_{i=1}^{J_T}\tilde{\varrho}_i.
\]
Then it follows from \eqref{est in Step1} and \eqref{est in step2} that
\begin{align}\label{identity in step3}
\hat{\mbb E}\left[f(X^y_T)\left(S_T-\tilde{S}_T\right)   \right]+\int_{\mbb D[0,\infty)}I_2(w')\mbb Q_y(dw')=\hat{\mbb E}\left[f(X^x_T)\bar{S}_T\right].
\end{align}
By the definition of $I_2$, we have
\[
\left|\int_{\mbb D[0,\infty)}I_2(w')\mbb Q_y(dw')\right|\le
\|f\|
\frac{\int_0^T\hat X_{s-}
\mathbb P_{x-y}(Y_s>\varepsilon/2)\,ds}{\hat{A}_T},
\]
where $Y$ is a CB process with initial value $x-y$. Moreover, the definition of $J$ and the compensation formula give
\[
\hat{\mbb E}[\tilde{S}_T]=\hat{A}_T^{-1}\int_0^T\hat{X}_{s-}\mbb P_{x-y}(Y_s>\varepsilon/2)ds.
\]
Then
combining those estimates with \eqref{identity in step3}, we have
\begin{align}\label{est in step3-1}
&|\hat{\mbb{E}}f(X^x_T)-\hat{\mbb{E}}f(X^y_T)|\cr
&\leq \|f\|\left(\hat{\mbb{E}}
\left|1-S_T\right|+\hat{\mbb{E}}
\left|1-\bar{S}_T\right|+2\hat{A}_T^{-1}\int_0^T\hat{X}_{s-}\mbb{P}_{x-y}(Y_s>\varepsilon/2)\,ds\right).
\end{align}
Let $\mathcal J_T:=\sigma\{J_s:0\leq s\leq T\}.$
Conditional on $\mathcal J_T$, the jump times $\{\tau_i\}$ are fixed,
whereas $\{\eta_i\}$ are independent with common distribution
$\frac{\nu(\cdot\cap B(\varepsilon))}{\nu(B(\varepsilon))}$.
Therefore, both sequences
$\{\bar\varrho_i\}$ and $\{\varrho_i\}$ are conditionally
independent given $\mathcal J_T$. For $\bar{\varrho}_i$, it is not hard to see that
\begin{align}\label{moment of bar varrho}
\hat{\mathbb E}
[\bar\varrho_i\mid\mathcal J_T]
=1,\qquad \hat{\mathbb E}
[\bar\varrho_i^2\mid\mathcal J_T]
=\frac{\nu(B(\varepsilon))}{\nu(B(\varepsilon/2))}.
\end{align}
For \(\varrho_i\), by using
$\nu(dr)=q(r)\,dr$ on $B(\varepsilon)$, we obtain
\begin{align}\label{first-moment of varrho}
&\hat{\mathbb E}
[\varrho_i\mid\mathcal J_T]\cr
&=
\frac{1}{\nu(B(\varepsilon/2))}
\int_0^{\varepsilon/2}Q_{\tau_i}(x-y,dz)
\int_{B(\varepsilon)}
\frac{q(r-z)}{q(r)}
\mathbf 1_{\{r-z\in B(\varepsilon/2)\}}\nu(dr)+
\int_{\varepsilon/2}^{\infty}Q_{\tau_i}(x-y,dz)\cr
&=
\int_0^{\varepsilon/2}Q_{\tau_i}(x-y,dz)
+\int_{\varepsilon/2}^{\infty}Q_{\tau_i}(x-y,dz)=1.
\end{align}
Thanks to Jensen's
inequality and the elementary inequality $(a+b)^2\le 2(a^2+b^2)$ for $a,b>0$, we get
\begin{align}\label{second-moment of varrho}
\hat{\mathbb E}
[\varrho_i^2\mid\mathcal J_T]
&\le
2+
\frac{2\nu(B(\varepsilon))}{\nu(B(\varepsilon/2))^2}
\int_{B(\varepsilon)}
\int_0^{\varepsilon/2}
\frac{q^2(r-z)}{q(r)}
\mathbf 1_{\{r-z\in B(\varepsilon/2)\}}
Q_{\tau_i}(x-y,dz)\,dr\cr
&\le
2+\frac{2\|q\|^2_{L^\infty(B(\varepsilon))}\nu(B(\varepsilon))}{\nu(B(\varepsilon/2))^2}
\int_{B(\varepsilon)}q(r)^{-1}\,dr:=\sigma_\varrho<\infty,
\end{align}
where $\|q\|_{L^\infty(B(\varepsilon))}<\infty$ is due to Condition \ref{condition absolute}.  Therefore, by \eqref{first-moment of varrho} and \eqref{second-moment of varrho},
\[
\begin{aligned}
&\hat{\mbb E}[S_T\mid\mathcal J_T]
=\frac{1}{\nu(B(\varepsilon))\hat{A}_T}\sum_{i=1}^{J_T}\hat{\mbb E}[\varrho_i\mid\mathcal J_T]
=\frac{J_T}{\nu(B(\varepsilon))\hat{A}_T},\cr
&\mrm{\hat{Var}}(S_T\mid\mathcal J_T)
=\frac{1}{\nu(B(\varepsilon))^2\hat{A}^2_T}\sum_{i=1}^{J_T}\mrm{\hat{Var}}(\varrho_i\mid\mathcal J_T)
\le\frac{J_T(\sigma_\varrho-1)}{\nu(B(\varepsilon))^2\hat{A}^2_T},
\end{aligned}
\]
where $\mrm{\hat{Var}}$ denote the variance under $\hat{\mbb P}$. By the conditional variance formula,
\begin{align}\label{variance of ST}
\hat{\mbb E}(1-S_T)^2
&=\hat{\mbb E}\left[\mrm{\hat{Var}}(S_T\mid\mathcal J_T)\right]
+\hat{\mbb E}\left[\left(1-\hat{\mbb E}[S_T\mid\mathcal J_T]\right)^2\right]\cr
&\le \frac{\sigma_\varrho-1}{\nu(B(\varepsilon))^2\hat{A}_T^2}\hat{\mbb E}[J_T]+\hat{\mbb E}\left[\left(1-\frac{J_T}{\nu(B(\varepsilon))\hat{A}_T}\right)^2\right]\cr
&=\frac{\sigma_\varrho-1}{\nu(B(\varepsilon))\hat{A}_T}
+\frac{1}{\nu(B(\varepsilon))\hat{A}_T}
=\frac{\sigma_\varrho}{\nu(B(\varepsilon))\hat{A}_T}.
\end{align}
Similarly, we have from \eqref{moment of bar varrho}  that
\begin{align}\label{variance of bar ST}
\hat{\mbb E}\left(1-\bar{S}_T\right)^2
\leq\frac{1}{\nu(B(\varepsilon/2))\hat{A}_T}.
\end{align}
Combining H\"older inequality under $\hat{\mathbb P}$, \eqref{est in step3-1}, \eqref{variance of ST} with \eqref{variance of bar ST}, we arrive at
\[
|\hat{\mbb{E}}f(X^x_T)-\hat{\mbb{E}}f(X^y_T)|
\leq \lambda_1\|f\|\hat{A}_T^{-1/2}
+\lambda_2\|f\|
\frac{\int_0^T\hat{X}_{s-}\mbb{P}_{x-y}(Y_s>\varepsilon/2)\,ds}{\hat{A}_T}
\]
for some constants $\lambda_1,\lambda_2>0$. Taking expectations on both sides gives
\begin{align}\label{total variation}
|\mbb{E}f(X^x_T)-\mbb{E}f(X^y_T)|
&\leq \mbb{E}\left|
\hat{\mbb{E}}f(X^x_T)-\hat{\mbb{E}}f(X^y_T)\right|\cr
&\leq \|f\|\left[
\lambda_1\mbb{E}\left[\hat{A}_T^{-1/2}\right]
+\lambda_2\mbb{E}\left[
\frac{\int_0^T\hat{X}_{s-}\mbb{P}_{x-y}(Y_s>\varepsilon/2)\,ds}
{\hat{A}_T}\right]\right].
\end{align}
Since $Y$ is a CB process with semigroup $(Q_t)_{t\ge0}$, for every $s>0$ we have
\begin{align}\label{tail prob of Ys}
\mbb{P}_{x-y}(Y_s>\varepsilon/2)
&=\mbb{P}_{x-y}\left(1-e^{-Y_s}>1-e^{-\varepsilon/2}\right)
\leq\frac{1}{1-e^{-\varepsilon/2}}\mbb{E}_{x-y}\left(1-e^{-Y_s}\right)\cr
&=\frac{1}{1-e^{-\varepsilon/2}}\left(1-e^{-(x-y)v_s(1)}\right)
\leq\frac{1}{1-e^{-\varepsilon/2}}(x-y)v_s(1),
\end{align}
where $t\mapsto v_t(1)$ is the unique solution of \eqref{2.2} with
$\lambda=1$. Note that
\begin{align}\label{ineq:3.24}
&\mbb{E}\left[
\frac{\int_0^T\hat{X}_{s-}\mbb{P}_{x-y}(Y_s>\varepsilon/2)\,ds}
{\hat{A}_T}\right]\cr
&=\mbb{E}\left[
\frac{\int_0^T\hat{X}_{s-}\mbb{P}_{x-y}(Y_s>\varepsilon/2)\,ds}
{\hat{A}_T}\mathbf{1}_{\{\hat{A}_T>\sqrt{T}\}}\right]
+\mbb{E}\left[
\frac{\int_0^T\hat{X}_{s-}\mbb{P}_{x-y}(Y_s>\varepsilon/2)\,ds}
{\hat{A}_T}\mathbf{1}_{\{\hat{A}_T\le\sqrt{T}\}}\right]\cr
&\leq
\frac{1}{\sqrt{T}}\int_0^T
\mbb{E}[\hat{X}_{s-}]\mbb{P}_{x-y}(Y_s>\varepsilon/2)\,ds+\sqrt{T}\,\mbb{E}\hat{A}_T^{-1}.
\end{align}
Since $\hat X$ is a subcritical CBI process satisfying \eqref{CBI_hat} with $\hat{\beta}<0$, Li \cite[(5.12)]{Li20} gives
$$
\sup_{s\ge0}\mathbb E[\hat X_{s-}]<\infty.
$$
Then by Proposition \ref{moment of 1/Y}, \eqref{tail prob of Ys} and \eqref{ineq:3.24}, for $T>1$,
\[
\mbb{E}\left[
\frac{\int_0^T\hat{X}_{s-}\mbb{P}_{x-y}(Y_s>\varepsilon/2)\,ds}
{\hat{A}_T}\right]\le
\frac{C_1|x-y|}{\sqrt{T}}\int_0^T v_s(1)\,ds+\frac{C_2}{\sqrt{T}}
\]
holds for some constant $C_1>0$.
Combining this bound with Proposition \ref{moment of 1/Y} and \eqref{total variation} gives
\[
\sup_{\|f\|\leq1}|\mbb{E}f(X^x_T)-\mbb{E}f(X^y_T)|
\leq
\frac{C_2}{\sqrt{T}}
\left(1+|x-y|\int_0^T v_s(1)\,ds\right)
\]
for some constant $C_2>0$. Moreover, for $0<u\le 1$ and $\xi\in B(\varepsilon)$, we have
\[
\Psi(u)=\int_0^\infty \left(e^{-u\xi}-1+u\xi\right)\nu(d\xi)
\ge\int_{B(\varepsilon)} \left(e^{-u\xi}-1+u\xi\right)\nu(d\xi)\ge c_0u^2
\]
for some $c_0>0$. Then by \eqref{2.2} and a comparison principle,
\[
v_t(1)\le \frac{1}{c_0t+1},\quad t\ge0.
\]
This implies that
\[
t^{-1/2}\int_0^t v_s(1) ds\le \frac{\log(c_0t+1)}{c_0\sqrt{t}}.
\]
This establishes Theorem \ref{main theorem}-(ii). In the subcritical case, by \eqref{change of vari of v}, we have
\[
\int_0^\infty v_s(1)\,ds
=\int_0^1\frac{u}{\Psi(u)}\,du\le \int_0^1 \frac{u}{-\beta u}du=-\frac1\beta,
\]
which proves Theorem \ref{main theorem}-(i).\qed

\subsection{Proof of Theorems \ref{main theorem2} and \ref{main theorem3}}

{\it Proof of Theorem \ref{main theorem2}.} Recall that
$\hat{\mbb{P}}(\cdot)=\mbb{P}(\cdot|\sigma(\hat{X}))$. Fix $T>0$ and
$x\geq y\geq0$. By the cluster decomposition,
\[
\begin{aligned}
X^x_T
&=Y^x_T+\hat{X}_T
+\int_0^T\int_0^{\hat{X}_{s-}}\int_{\mathbb D}w(T-s)M(ds,du,dw)\\
&=Y^y_T+\hat{X}_T
+\int_0^T\int_0^{\hat{X}_{s-}}\int_{\mathbb D}w(T-s)M(ds,du,dw)
+\tilde{Y}_T,
\end{aligned}
\]
where $\tilde{Y}=Y^x-Y^y$, and $\tilde{Y},Y^y,\hat{X}$ and $M$ are
independent. For any bounded Borel function $f$,
\begin{align}\label{bound of expon ergod}
\left|\mbb{E}f(X^x_T)-\mbb{E}f(X^y_T)\right|
\leq 2\|f\|\mbb{P}(\tau_1>T)+\left|\mbb{E}\left[
(f(X^x_T)-f(X^y_T))\mbf{1}_{\{\tau_1\leq T\}}\right]\right|,
\end{align}
where $\tau_1$ is
the first cluster jump whose size belongs to $B(\varepsilon)$, and
\begin{align}\label{tail prob of tau1}
\hat{\mbb{P}}(\tau_1>s)
=\exp\left\{-\nu(B(\varepsilon))\int_0^s\hat{X}_r\,dr\right\},\quad s>0.
\end{align}
Conditionally on $\hat X$, the first-jump density is
$\nu(B(\varepsilon))\hat X_{s-}\hat{\mbb{P}}(\tau_1>s)$. Using Condition \ref{condition basic}, we have
\begin{align}\label{ineq:3.17}
&\left|\mbb{E}\left[
(f(X^x_T)-f(X^y_T))\mbf{1}_{\{\tau_1\leq T\}}\right]\right|\cr
&\leq
\mbb{E}\int_0^T\nu(B(\varepsilon))\hat X_{s-}\hat{\mbb{P}}(\tau_1>s)
\left|\int_0^\infty
P_{T-s}f(X^y_{s-}+\tilde{Y}_{s-}+r)\bar{\nu}_{\varepsilon}(dr)\right.\cr
&\qqquad\qqquad\qqquad\left.
-\int_0^\infty
P_{T-s}f(X^y_{s-}+r)\bar{\nu}_{\varepsilon}(dr)\right|ds\cr
&\leq
K_{\varepsilon}\nu(B(\varepsilon))(x-y)\|f\|
\int_0^T \mbb E[\hat{X}_{s-}\hat{\mbb P}(\tau_1>s)]e^{\beta s}\,ds\cr
&\leq
K_{\varepsilon}(x-y)\|f\|,
\end{align}
where the last inequality is due to $\beta<0$ and
\[
\int_0^T\nu(B(\varepsilon))\hat{X}_{s-}\hat{\mbb P}(\tau_1>s)ds=1-\hat{\mbb P}(\tau_1>T)<1.
\]

Moreover, by Filipovi\'c \cite[Theorem 5.3]{Fil01} and \eqref{tail prob of tau1},
\begin{align}\label{ineq:3.18}
\mbb P(\tau_1>T)
=\exp\left[-b\int_0^T\hat{V}_s(\nu(B(\varepsilon)))ds\right],
\end{align}
where $t\to\hat{V}_t(\lambda)$ is the unique solution of
\[
\frac{\partial \hat{V}_t(\lambda)}{\partial t}=-\hat{\Psi}(\hat{V}_t(\lambda))+\lambda,\quad \hat{V}_0(\lambda)=0
\]
and $\hat{\Psi}$ is given by \eqref{mechanism of hat X}. Similarly as in the proof of Proposition \ref{moment of 1/Y}, there exist $\hat{\beta}_1,\hat{\beta}_2>0$ such that
\[
\hat{\Psi}(\lambda)\le \hat{\beta}_1\lambda+\hat{\beta}_2\lambda^2,\quad \hat{V}_t(\lambda)\ge\frac{\lambda\left(1-e^{-D_\lambda t}\right)}{D_\lambda},\quad \lambda\ge0,
\]
where \(D_\lambda=(\hat{\beta}_1^2+4\hat{\beta}_2\lambda)^{1/2}\). Then,
\[
\int_0^T\hat{V}_s(\lambda)ds\ge
\frac{\lambda}{D_\lambda}\left(T-\frac{1-e^{-D_\lambda T}}{D_\lambda}\right)\ge\frac{\lambda T}{D_\lambda}-\frac{\lambda}{D_\lambda^2}.
\]
Together with \eqref{ineq:3.18}, this implies that
\begin{align}\label{ineq:3.19}
\mbb P(\tau_1>T)\le C_\varepsilon e^{-\kappa_\varepsilon T}, \quad
C_\varepsilon=\exp\left(\frac{b\nu(B(\varepsilon))}{D^2_{\nu(B(\varepsilon))}}\right),\quad
\kappa_\varepsilon=\frac{b\nu(B(\varepsilon))}{D_{\nu(B(\varepsilon))}}.
\end{align}
Therefore, by \eqref{bound of expon ergod}, \eqref{ineq:3.17} and \eqref{ineq:3.19},
\[
\left|\mbb{E}f(X^x_T)-\mbb{E}f(X^y_T)\right|\le 2C_\varepsilon e^{-\kappa_\varepsilon T}\|f\|
+K_\varepsilon (x-y)\|f\|.
\]
An application of the Markov property gives
\[
\begin{aligned}
\left|\mbb{E}f(X^x_{2T})-\mbb{E}f(X^y_{2T})\right|
&\leq\mbb{E}\left|P_Tf(X^x_T)-P_Tf(X^y_T)\right|\cr
&\leq2C_\varepsilon\|f\|e^{-\kappa_\varepsilon T}
+K_\varepsilon\|f\|\mbb{E}(X^x_T-X^y_T)\cr
&=2C_\varepsilon\|f\|e^{-\kappa_\varepsilon T}
+K_\varepsilon\|f\|(x-y)e^{\beta T}.
\end{aligned}
\]
Finally, the desired result follows since $\beta<0$. \qed

{\it Proof of Theorem \ref{main theorem3}.} For each $n\ge1$, define
\[
\hat{\Psi}_n(\lambda):=\left(-\beta+\int_{B(\varepsilon_n)}\xi\nu(d\xi)\right)\lambda
+\int_{B^c(\varepsilon_n)}
\left(e^{-\lambda\xi}-1+\lambda \xi\right)\nu(d\xi),\quad \lambda\ge0.
\]
Fix some $\delta\in(0,1)$. We have
\[
\begin{aligned}
\int_{B^c(\varepsilon_n)}
\left(e^{-\lambda\xi}-1+\lambda \xi\right)\nu(d\xi)
&\le \frac{\lambda^2}{2}
\int_{B^c(\varepsilon_n)\cap(0,\delta)}\xi^2\nu(d\xi)
+\lambda\int_{B^c(\varepsilon)\cap[\delta,\infty)}\xi\nu(d\xi)\cr
&\le \frac{\lambda^2}{2}
\int_0^\delta\xi^2\nu(d\xi)
+\lambda\int_\delta^\infty\xi\nu(d\xi)
\end{aligned}
\]
due to the facts that
\[
e^{-z}-1+z\le z^2/2,\quad e^{-z}-1+z\le z,\quad z\ge0.
\]
Then
\[
\begin{aligned}
\hat{\Psi}_n(\lambda)&\le
\left[-\beta+\int_{B(\varepsilon_n)}\xi\nu(d\xi)+\int_\delta^\infty \xi\nu(d\xi)\right]\lambda+\frac{1}{2}\left[\int_0^\delta \xi^2\nu(d\xi)\right]\lambda^2\cr
&\le \left[c_*+\int_{B(\varepsilon_n)}\xi\nu(d\xi)\right]\lambda+c_{**}\lambda^2,
\end{aligned}
\]
where
\[
c_*=1-\beta+\int_\delta^\infty \xi\nu(d\xi),\quad c_{**}=1+\frac{1}{2}\int_0^\delta \xi^2\nu(d\xi).
\]
Set $\alpha_n:=\nu(B(\varepsilon_n))$. Similarly as in the proof of Theorem \ref{main theorem2}, the associated first cluster jump $\tau_1^{(n)}$ satisfies
\[
\mbb P\left(\tau^{(n)}_1>T\right)
=\exp\left[-b\int_0^T \hat{V}_s^{(n)}(\alpha_n)ds\right],
\]
and
\[
\hat{V}_s^{(n)}(\alpha_n)
\ge\frac{\alpha_n(1-e^{-D_{\alpha_n}t})}{D_{\alpha_n}},\quad D^2_{\alpha_n}=\left(c_*+\int_{B(\varepsilon_n)}\xi\nu(d\xi)\right)^2+4c_{**}\alpha_n.
\]
It follows that
\[
\mbb P\left(\tau_1^{(n)}>T\right)
\le C_{\varepsilon_n}e^{-\kappa_{\varepsilon_n}T},\qquad
C_{\varepsilon_n}=\exp\left(\frac{b\alpha_n}{D^2_{\alpha_n}}\right),\qquad
\kappa_{\varepsilon_n}=\frac{b\alpha_n}{D_{\alpha_n}}.
\]
For any $h>0$,
\[
\int_{B(\varepsilon_n)}\xi\nu(d\xi)=\int_{B(\varepsilon_n)\cap(0,h)}\xi\nu(d\xi)+\int_{B(\varepsilon_n)\cap[h,\infty)}\xi\nu(d\xi)\le h\alpha_n+\int_h^\infty \xi\nu(d\xi),
\]
which implies that
\[
\lim_{n\to\infty}\alpha_n^{-1}\int_{B(\varepsilon_n)}\xi\nu(d\xi)=0
\]
and so
\begin{align}\label{strong feller constant}
\lim_{n\to\infty}\kappa_{\varepsilon_n}=\infty.
\end{align}

Suppose that $\{r_k\}\subset\mbb{R}_+$ is a sequence such that $\lim_{k\rightarrow\infty}r_k=r_0\in\mbb{R}_+$. Following a similar argument as in the proof of Theorem \ref{main theorem2}, for any $t>0$ and $n\ge1$,
\[
\limsup_{k\rightarrow\infty}\|P_t(r_k,\cdot)-P_t(r_0,\cdot)\|_{\rm{Var}} =\limsup_{k\rightarrow\infty}\sup_{\|f\|\le1}\left|\mbb{E}f(X^{r_k}_{t})-\mbb{E}f(X^{r_0}_{t})\right|\le 2C_{\varepsilon_n}e^{-\kappa_{\varepsilon_n}t/2}.
\]
Thanks to \eqref{strong feller constant} and the fact that
$C_{\varepsilon_n}\le \exp(\frac{b}{4c_{**}})<\infty$, we finish the proof. \qed

\bigskip
\noindent

{\bf Acknowledgements.}
The research of Shukai Chen is supported
    by the National Key R\&D Program of China
(No.2022YFA1006003), NSFC grant of China
(No.12401167),
Fujian Provincial Natural Science Foundation of China (No.2024J08050). The research of Chunhua Ma is supported by NSFC grant of China (No.11871032). We are grateful to Pei-Sen Li and
Jian Wang for their helpful comments.

	\begin{singlespace}
		
	\end{singlespace}

	\vskip 0.2truein
	\vskip 0.2truein

\bigskip

\noindent{\bf Shukai Chen:}  School of Mathematics and Statistics, Fujian Normal University,
Fujian, P. R. China Email: {\texttt skchen@fjnu.edu.cn}

\bigskip
\noindent{\bf Chunhua Ma:}  School of Mathematical Sciences and LPMC, Nankai University, Tianjin, P. R. China  Email: {\texttt mach@nankai.edu.cn}

\end{document}